\documentclass[9pt]{amsart}
\usepackage{natbib}
\usepackage[applemac]{inputenc}
\usepackage[T1]{fontenc}
\usepackage{hyperref}
\usepackage{graphicx}
\usepackage{setspace}
\usepackage{amsmath}
\usepackage{amsfonts}
\usepackage{amsthm}
\usepackage{booktabs}
\usepackage{framed}
\usepackage{bbding}
\usepackage{ifthen}
\usepackage{color}
\usepackage{enumitem}
\usepackage[format=default,font=footnotesize,labelfont=bf]{caption}
\usepackage[paper=a4paper,left=30mm,right=30mm,top=25mm,bottom=25mm]{geometry}
\renewcommand{\P}{\mathbb{P}}
\renewcommand{\d}{\mathrm{d}}
\newcommand{\hawkesorder}[1]{} 

\DeclareMathOperator{\Pois}{\mathrm{Pois}}

\DeclareMathOperator{\Z}{\mathbb{Z}}
\DeclareMathOperator{\N}{\mathbb{N}}

\DeclareMathOperator{\E}{\mathbb{E}}
\DeclareMathOperator{\R}{\mathbb{R}}

\DeclareMathOperator{\Var}{\mathrm{Var}}
\DeclareMathOperator{\Cov}{\mathrm{Cov}}

\renewcommand{\[}{\left[}
\renewcommand{\]}{\right]}

\renewcommand{\d}{\mathrm{d}}

\renewcommand{\N}{\mathbb{N}}
\newtheorem{definition}{Definition}
\newtheorem{theorem}{Theorem}

\newtheorem{conjecture}{Conjecture}
\newtheorem{example}{Example}

\renewcommand\footnotemark{}
\theoremstyle{definition}

\begin{document}
\date{\today}
\title{A note on critical Hawkes processes}
\author{Matthias Kirchner}
\maketitle

\begin{abstract}
Let $F$ be a distribution function on $\R$ with $F(0) = 0 $ and density $f$. Let $\tilde{F}$ be the distribution function of $X_1 - X_2$, $X_i\sim F,\, i=1,2,\text{ iid}$. We show that for a critical Hawkes process with displacement density (= `excitement function' = `decay kernel') $f$, the random walk induced by $\tilde{F}$ is necessarily transient. Our conjecture is that this condition is also sufficient for existence of a critical Hawkes process.
Our train of thought relies on the interpretation of critical Hawkes processes as cluster-invariant point processes. From this property, we identify the law of critical Hawkes processes as a limit of independent cluster operations. We establish uniqueness, stationarity, and infinite divisibility. Furthermore, we provide various constructions: a Poisson embedding, a representation as Hawkes process with renewal immigration, and a backward construction yielding a Palm version of the critical Hawkes process. We give specific examples of the constructions, where $F$ is regularly varying with tail index $\alpha\in(0,0.5)$. Finally, we propose to encode the genealogical structure of a critical Hawkes process with Kesten (size-biased) trees. The presented methods lay the grounds for the open discussion of multitype critical Hawkes processes as well as of critical integer-valued autoregressive time series. \\\begin{bfseries}Keywords\end{bfseries}: Hawkes process, critical cluster field, cluster invariance, branching random walk, renewal theory, regular variation, Kesten tree.
 \end{abstract}

\section{Introduction}
Let $(\Omega, \mathcal{F},\P)$ be a probability space. Let $\mathcal{X}\in\mathcal{B}(\R^d)$ for some $\d\in\N$, and let $\mathcal{B}(\mathcal{X})$ ($\mathcal{B}_b(\mathcal{X})$)
denote the (bounded) Borel sets of $\mathcal{X}$.
A \emph{point process on $\mathcal{X}$} is a measurable mapping $N:(\Omega, \mathcal{F})\to(M_p, \mathcal{M}_p)$, where $M_p$ denotes the space of locally finite counting measures on $\mathcal{X}$ and $\mathcal{M}_p$ is the smallest $\sigma$-algebra of $M_p$  such that all mappings $M_p\ni m\mapsto m(B),\,B\in\mathcal{B}_b(\mathcal{X}),$ are measurable. For any point process $N$ on $\R$, set $\mathcal{F}\supset\mathcal{H}_x^N:= \sigma\big(\{\omega\in\Omega: N(\omega)(B) = n\}:\ n\in \N_0, B\in \mathcal{B}_b((-\infty,x] )\big),\, x\in\R$. The filtration $(\mathcal{H}_x^N)$ is the \emph{history of $N$}. We say that $\lambda(\cdot)$ is an \emph{$(\mathcal{H}_x^N)$-intensity} of a point process $N$ if $\lambda(\cdot)$ is adapted to $(\mathcal{H}_x^N)$, is almost surely locally integrable, and
\begin{align}
\E \Big[N\big((a, b]\big) 1_A\Big] = \E \left[\int_{(a, b]}\lambda(x)\mathrm{d} x1_A\right] ,\quad a < b, A\in \mathcal{H}_a^{N}. \label{G:eq:H-intensity}
\end{align}
If  $\P[\cap_{k = 1}^n \{N(B_k) = n_k\} ] =  \P[\cap_{k = 1}^n \{N(B_k+x) = n_k\} ]$, $x\in\R$, $n_k\in\N$, $B_k\in\mathcal{B}(\R), k= 1,2,\dots, n\in\N,$ then $N$ is \emph{stationary}. 
If $\P[\cap_{t\in\R}\{N(\{t\}) \leq 1\} ] = 1$, then $N$ is \emph{simple}. 
For a simple and stationary point process $N$, we have that
$\E N(B)= \lambda\int_B\mathrm{d}x$ for some $\lambda \in[0,\infty]$. We call
$\lambda$ \emph{(average) intensity} and $\E N(\cdot)$ \emph{mean measure} of $N$. A \emph{subcritical Hawkes process} is a simple stationary point process $N$ admitting the $(\mathcal{H}_x^{N})$-intensity 
\begin{align}
\lambda(x) = \eta + m\int_{(-\infty,t)}f(x - y) N(\mathrm{d} y),\quad x\in\R,\label{G:eq:def}
\end{align}
where $\eta\in(0,\infty)$ denotes the \emph{immigration intensity}, $m\in[0,1)$ the \emph{branching coefficient}, and $f$ the \emph{displacement density} supported by $[0,\infty)$. It is well known that \eqref{G:eq:def} specifies a unique distribution on $(M_p, \mathcal{M}_p)$; see \citet{hawkes71b}. The construction of a Hawkes process involves a branching mechanism: the first step of the construction is a Poisson random measure on $\R$ with intensity $\eta$. Each of these \emph{immigrants} is the ancestor of a Hawkes family and has  $\Pois(m)$ \emph{children}. Each of these children again has $\Pois(m)$ children etc. The distance between parent and child is modeled by the displacement density $f$. All these offspring and displacement operations are independent. Each immigrant with all its \emph{descendants} forms a \emph{Hawkes family}. The Hawkes process is the random measure that counts the superposition of these Hawkes families in a given set. In other words, the Hawkes process is a sum of occupation measures of  infinitely many particles performing (one-sided) branching random walks. Taking expectations on either side of \eqref{G:eq:def} and applying stationarity, we note that the intensity of the process is $\lambda=\eta / (1 - m)$. We obtain the same result by multiplying the intensity of the immigrants, that is, $\eta$, with the expected number of points in a Hawkes family, that is, $(1-m)^{-1}$. In \citet{bremaud01}, the authors consider a set of Hawkes processes $\{N^{(m)}\}_{m\in(0,1)}$, where the immigration intensity depends on the branching coefficient $m$ by setting $\eta_m := (1 - m)\lambda,\, m\in[0,1),$ for some fixed $\lambda >0$. Obviously, the  intensity for all of these processes equals $\lambda$. The cited paper considers weak limits $N^{(m)}$ as $m\uparrow 1$. Note that increasing $m$ corresponds to thinning immigrants and---at the same time---enlarging the remaining Hawkes families. Theorem 1 in the same reference states that if
\begin{align}
  \sup\limits_{x\geq 0} f(x) x ^{1 + \alpha}\leq R\quad \text{and} \quad \lim\limits_{x\to\infty}f(x)x^{1 + \alpha} = r\label{G:eq:bremaud_conditions}
\end{align}
for some constants $r,R\in(0,\infty)$ and $\alpha\in(0,0.5)$,
then $N^{(m)}$ converges weakly to a point process $N^{(1)}$ with average intensity $\lambda$ and $(\mathcal{H}^{N^{(1)}}_x)$-intensity
\begin{align}
\lambda(x) = \int_{(-\infty, x)} f(x - y) N^{(1)}(\mathrm{d}y).\label{G:eq:critical_hawkes_equation}
\end{align}
In view of the branching construction, the existence of such a critical Hawkes process $N^{(1)}$ is somewhat paradoxical: for an arbitrary point of a Hawkes process, consider the number of its children, of its grandchildren, of its great-grandchildren, etc. This sequence forms a Galton--Watson process with offspring distribution $\Pois(m)$. However, it is well known that not only subcritical ($m < 1$) but also critical ($m = 1$) Galton--Watson processes die out almost surely; for example, see Theorem 6.1 in \citet{harris63}. Consequently, in a critical Hawkes process, any realized point will almost surely have only a finite number of descendants. For illustration, consider $F(0-) = 0$ and $F(0) = 1$ for the displacement distribution (for the sake of example, we disregard simplicity of the process). In this case, the position of each point coincides with the position of its descendants. Therefore, if the immigrants are thinned, then all descendants are bound to vanish---and consequently the whole process with them. In fact, Proposition 1 in \citet{bremaud01} shows that $\int x\d F(x) < \infty$ already has the same effect: if the expectation of the displacements is finite, no nontrivial solution to \eqref{G:eq:critical_equation} can exist. In contrast, the conditions in \eqref{G:eq:bremaud_conditions} guarantee that the displacements are balanced in such a way that the Hawkes families grow larger and larger to fill the larger and larger gaps between the immigrants. We found that the resulting object, the critical Hawkes process, is related to many well-studied topics:
\begin{enumerate}[label = \alph*)]
\item \emph{Cluster invariance}: we observe that any critical Hawkes process must have this distributional property. This property opens the door to a complete theory presented in Chapter 12 of \citet{matthes78};
\item \emph{Renewal theory}: we observe that critical Hawkes processes can be interpreted as  Hawkes processes with renewal immigration, where the interarrival times have infinite mean. Thus, in the long run, we can no longer observe immigrants. This construction can be studied with standard tools of renewal theory as presented, e.g., in Chapter 8 of \citet{bingham87};
\item  \emph{Backward trees}: given a critical Hawkes process, we may pick any point and walk backwards the genealogical structure. In such a manner, we can (re-)construct the process and circumvent the lack of a root event (like immigrants). This method is related to the `method of backwards trees' introduced in \citet{kallenberg77};
\item \emph{Kesten trees} or \emph{size-biased trees}: the lack of immigrants---that is, root nodes---and the non-extinction property make branching representations of critical Hawkes processes not straightforward. We find that the `Kesten tree' (see, e.g., \citet{lyons95})---a special Galton-Watson tree with size-biased offspring distribution---may be used to encode the genealogical structure of a critical Hawkes process.
\end{enumerate}
We hope that these connections will help to complete the understanding of critical Hawkes processes. In particular, the methods presented offer possible directions for the open discussion of multitype critical Hawkes processes as well as of integer-valued autoregressive time series. Finally note that the present paper gives rise to the following very general conjecture that connects the existence of critical Hawkes processes with the recurrence/transience dichotomy of random walks. We say that a distribution is \emph{transient} (\emph{recurrent}) if the random walk induced by this distribution is transient (recurrent). Furthermore, for any distribution $F$, let $\tilde{F}$ denote the distribution of $X_1 - X_2$, $X_i\sim F,\, i=1,2,\text{ iid}$. We call $\tilde{F}$ the \emph{symmetrized version of F}. Note that $\tilde{F}:\R\to [0,1],\, x\mapsto \int F(x + y) \mathrm{d} F(y)$. 
\begin{conjecture}\label{G:conj:1}
Let $F$ be an absolutely continuous distribution function with $F(0) = 0 $ and density $f$. Let $\tilde{F}$ be the symmetrized version of $F$ and let $\lambda >0$. A critical Hawkes process $N$ with displacement density $f$ and average intensity $\lambda$ exists if and only if the symmetrized version of the displacement distribution $\tilde{F}$ is transient.
In this case, \eqref{G:eq:critical_hawkes_equation} specifies a unique, stationary, and infinitely divisible distribution.
\end{conjecture}

Note that Conjecture~\ref{G:conj:1} includes Proposition 1 of \citet{bremaud01} stating the necessary condition $\int x\mathrm{d}F(x) = \infty$ for existence of a critical Hawkes process. Indeed:
if $\int x\mathrm{d}F(x) < \infty$, we also have $\int x\mathrm{d}\tilde{F} = 0$. Consequently, by the Chung--Fuchs theorem (see \citet{feller71}, Section~XVIII.6, Lemma 1) $\tilde{F}$ is recurrent. 

\section{Perspectives on critical Hawkes processes}\label{G:sec:perspectives_on_critical_hawkes_processes}
For any distribution $F$, $m\in(0,\infty)$, and any point process $N\sim L$ with points $\{T_n\}_{n\in I}$, $I\subset\Z$, we define the  \emph{(Poisson) clustering operation}
\begin{align}\label{G:eq:cluster_operation}
[F,m]\star \big\{T_n\}_{n\in I} := \bigcup_{n \in I }\bigcup_{k = 1}^{K_{n}} \big\{T_n + X_{n,k}\big\},
\end{align} 
where $\{K_n, X_{n,k}:\, k\in\N,n\in\Z\}$ are independent random \emph{cluster variables} (also independent of ${N}$) with $K_n\sim\Pois(m)$ and $X_{n,k}\sim F.$ We denote the point process resulting from \eqref{G:eq:cluster_operation} by $N_{[F,m]} $ and its distribution by $ L_{[F,m]}$. We call $[F,m]$ a \emph{cluster field induced by $F$}.
We may construct a univariate subcritical Hawkes process $N$ with immigration intensity $\eta\in(0,\infty)$, branching coefficient $m\in(0,1)$, and displacement distribution $F$ by applying this clustering operation: 
\begin{align}
N := \sum\limits_{g\geq 0}N^{(g)}
\text{ with } 
N^{(0)}\sim\mathrm{PRM}(\eta)\quad
\text{and} \quad N^{(g)}:= N^{(g-1)}_{ [F,m]_g},\quad g\in\N,
\label{G:eq:hawkes_construction}
\end{align}
where the clustering operations `$[F,m]_g\star\cdot$' are applied independently over $g\in\N$. For an absolutely continuous distribution $F$ with $F(0)=0$ and density $f$, one can show that the point process $N$ as in \eqref{G:eq:hawkes_construction} solves \eqref{G:eq:def}. If $m=1$, any construction as in \eqref{G:eq:hawkes_construction} (with $\eta>0)$ would yield an infinite average intensity of the limit process $N$. However, one can argue from \eqref{G:eq:critical_hawkes_equation} that the case $m=1$ can also be represented in terms of the cluster operation: 
 \begin{definition}\label{G:def:critical_hawkes}
 Let $F$ be a distribution on $\R$ with $F(0) = 0$ and $[F,1]$ the induced cluster field. Assume that $N$ is an ergodic 
 solution to
 \begin{align}
N =  N_{[F,1]}\label{G:eq:critical_equation}
\end{align}
with finite and constant average intensity $\lambda>0$.
Then we call $N$ a \emph{critical $(F,\lambda)$-Hawkes process}.
 \end{definition} 
If $F$ is absolutely continuous with density $f$, the critical cluster operation can be interpreted as attaching an inhomogeneous Poisson process with intensitiy $f(\cdot - T_n)$ to each point $T_n$. Thus, one can show that the critical $(F,\lambda)$-Hawkes process solves \eqref{G:eq:critical_hawkes_equation}. Vice versa, \eqref{G:eq:critical_equation} is more general than \eqref{G:eq:critical_hawkes_equation} in that the displacement distribution is not necessarily absolutely continuous. Also note that \eqref{G:eq:critical_equation} specifies a unique parent point for every point $T_n$ of $N$:  indeed, for each $n\in \Z$, there exist $n'\in\Z$ and $k\in\N$ such that $k\leq K_{n'}$ and $T_{n'} + X_{n',k} = T_n$, where the random variables $K_{n'}$ and $X_{n',k}$ stem from the clustering operation; see \eqref{G:eq:cluster_operation}. That is, the critical Hawkes process is `eating its own tail'. The trivial---yet crucial---observation is that \eqref{G:eq:critical_equation} also holds in distribution. 

 \subsection{Critical cluster fields}\label{G:eq:critical_cluster_fields}
 For any distribution $F$ on $\R$, denote by $[F]:=[F,1]$ the \emph{critical (Poisson) cluster  field} induced by $F$. Chapter 12 in \citet{matthes78} (MKM) discusses (distributions of) nontrivial stationary point processes  $N\sim L$ with the property 
 \begin{align}\label{G:eq:cluster_invariance}
 L = L_{[F]}
 \end{align} 
(in even higher generality).
If \eqref{G:eq:cluster_invariance} holds, then $N$ and its distribution $L$ are called \emph{cluster invariant with respect to $[F]$}. Furthermore, $F$ is called \emph{stable} if such a distribution $L$ exists.
From \eqref{G:eq:critical_equation}, we get that critical Hawkes processes are obviously cluster invariant with respect to the cluster field induced by their displacement distribution. Thus, we obtain several necessary conditions for the existence of critical Hawkes processes as corollaries from standard results:
 \begin{theorem}\label{G:thm:necessary_conditions}
 Let $F$ be a distribution on $[0,\infty)$ and $\lambda > 0$. Assume that a critical $(F,\lambda)$-Hawkes process $N$ as in Definition~\ref{G:def:critical_hawkes} exists. Then the following holds:
\begin{enumerate}[label = \alph*)]
\item\label{G:item:dbn_properties} Definition \ref{G:def:critical_hawkes} specifies a unique, infinitely divisible, and stationary distribution $H$ on $(M_p,\mathcal{M}_p)$.

\item\label{G:item:dbn_limit}   Let $L$ be the distribution of a Poisson random measure with finite average intensity $\lambda$. For $g\in\N$, denote $g$ independent clustering operations by `$[F^{[g]}]\star\cdot$'. Then, as $g\to \infty$, $L_{[F^{[g]}]}$ converges weakly to $H$. 

\item \label{G:item:existence_criterion} The symmetrized displacement distribution $\tilde{F}$ is transient.
\end{enumerate}
\end{theorem}
\begin{proof}
We note that, by definition, any possible distribution $H$ of a critical $(F,\lambda)$-Hawkes process is cluster invariant with respect to $[F]$, has bounded average intensity $\lambda$, and is ergodic. In particular, $F$ is stable. The statements of the theorem then follow from results in MKM: from Theorem 12.1.4.\ in MKM, we get that $L_{[F^{[g]}]}$ as in \ref{G:item:dbn_limit} converges weakly to an infinitely divisible limit distribution. Stationarity of any possible $H$ follows from Proposition 12.4.7.\ in MKM. From Theorem 12.4.1.\ in MKM, we get that $H$ coincides with the limit distribution of $L_{[F^{[g]}]}$. Consequently, $H$ is unique and infinite divisible. We have established \ref{G:item:dbn_properties} and \ref{G:item:dbn_limit}. We obtain \ref{G:item:existence_criterion}  from Theorem 12.6.6.\ in MKM if the variance of the offspring distribution is neither zero nor infinity. This assumption obviously holds for our $\Pois(1)$ offspring.
\end{proof}
Next to the reference MKM the reader is referred to \citet{kallenberg77} and
 Chapter~13.5 in \citet{daley03}, in particular, Proposition 13.5.II.\ therein. Theorem~\ref{G:thm:necessary_conditions} shows that $H$ can be seen as a steady distributional state that is reached by iterating clustering operations. Note that existence of a critical Hawkes process together with the assumptions on the average intensity \emph{imply} stationarity. 
The observation of cluster invariance only yields \emph{necessary} conditions for the existence of a critical Hawkes process. We now turn towards possible constructions to explore \emph{sufficient} existence conditions.

\subsection{Poisson embedding}\label{G:sec:poisson_embedding}
From Theorem~\ref{G:thm:necessary_conditions}~\ref{G:item:dbn_limit}, a straightforward  construction of a solution to \eqref{G:eq:critical_equation} and thus of a critical $(F,\lambda)$-Hawkes process would be to start with a Poisson random field with intensity $\lambda$, and then iterating the clustering operations `$[F]\star\cdot$'. However, we cannot hope that the resulting sequences $N^{(g)}(B),\, B\in \mathcal{B}_b(\R)$ will converge \emph{almost surely} to a nontrivial result when we apply independent (and in particular new) cluster variables at each new step. That is, the clustering operations of possible construction steps have to depend on each other. When the displacement distribution has a density $f$, we propose the following construction based on `Poisson embedding' similar to the (subcritical) Hawkes construction in \citet{bremaud96} or Chapter 6.3 in \citet{liniger09}: let $\mathcal{N}: \Omega \to (\R\times \R_{\geq 0})$ be a Poisson random measure with intensity $1$ on $\R^2$. We call $\mathcal{N}$ the \emph{driving process}. For $\lambda>0$, set $N^{(0)}(B)  := \int_{B}\mathcal{N}\big(\mathrm{d} x\times (0, \lambda]\big),\, B\in\mathcal{B}(\R)$, and, for $g\in\N$, recursively define point processes $N^{(g)}$ by
\begin{align}\label{G:eq:poisson_embedding}
\lambda^{(g)}(x) := \int_{(-\infty, x)} f(x - y) N^{(g-1)}(\mathrm{d} y),\quad
N^{(g)}(B)  := \int_{B}\mathcal{N}\big(\mathrm{d} x\times(0, \lambda^{(g)}(x)]\big),\quad B\in\mathcal{B}(\R).
\end{align}
Obviously, the average intensity equals $\lambda$ for all $N^{(g)},\, g\in\N_0$.
 The construction steps are very similar to the clustering operations `$[F]\star\cdot$'. Note however, that the clustering operations are not independent over $g\in\N$. In addition, note that the displacements for $N^{(g)}$ also depend on the positions of $N^{(g-1)}$. Thus, the marginal distributions of the point process sequence $(N^{(g)})$ are similar but not equal to $(L_{[F^{[g]}]})$, when $L$ denotes the starting Poisson random field $N^{(0)}$. The symmetrized distribution $\tilde{F}$ comes into play when calculating second moment measures of $N^{(g)}$ recursively. We think that the transience condition in Conjecture~\ref{G:conj:1} guarantees $\sup_{g\in\N}\Var(N^{(g)}(B)) <\infty,\, B\in\mathcal{B}_b(\R)$, and thus non-triviality of the potential limit $N^{(\infty)}$. We summarize the above: 
 
 \begin{conjecture}
 For any absolutely continuous displacement distribution $F$, $\lambda > 0$, and $B\in\mathcal{B}_b(\R)$, $N^{(g)}(B)$ as in \eqref{G:eq:poisson_embedding} converges almost surely to a nonnegative integer as $g\to\infty$. These limits define a point process $N^{(\infty)}$. If $\tilde{F}$ is transient, then the average intensity of $N^{(\infty)}$ equals $\lambda$, otherwise it equals 0. Furthermore, the processes $\lambda^{(g)}$ converge almost surely pointwise to a limit $\lambda^{(\infty)}$ such that $\lambda^{(\infty)}(\cdot)$ is an $\mathcal{F}^{N^{(\infty)}}$-intensity.
 \end{conjecture}

\subsection{Renewal immigration}\label{G:sec:renewal_immigration}
The average intensity of a subcritical Hawkes process equals the intensity of the immigrants times the expected number of points in a Hawkes family. Consequently, if $m = 1$, we need zero immigration intensity to obtain a locally finite mean measure. This can be thought of as immigrants stemming from a `stationary renewal process on $\R$ with infinite interarrival expectation'. We know from the Renewal Theorem (for example, see (1.9) in Chapter XI of \citet{feller71}), that the probability for observing such an immigrant in a finite interval will be zero. 
\begin{example}\label{G:ex:1}
Let $F$ be an absolutely continuous distribution on $[0,\infty)$ with infinite mean. 
 For any $c\geq0$, consider the truncated distribution $F_{c}(x) :=\P\[1_{X\leq c}X \leq x\],\, t\in\R$, where $X\sim F$, as well as the truncated mean
$$
\mu: [0,\infty)\ni c\mapsto\int x\mathrm{d}F_{c}(x).
$$ 
The function $\mu$ is non-decreasing and continuous with $\lim_{c\to\infty}\mu(c) = \infty$. Thus, we may define $\mu^{\leftarrow}:[0,\infty)\ni y\mapsto \inf\{c\geq 0:\ \mu(c) = y\},\ y\geq 0.$ Set $c:[0,1)\ni m\mapsto \mu^{\leftarrow}((1 - m)^{-1})$. The function $c$ is increasing and $\lim_{m\uparrow 1}c(m) = \infty$. For $m\in[0,1)$, let $N^{(m)}$ be a (subcritical) Hawkes process with branching coefficient $m\in[0,1)$, displacement distribution $F$, and stationary renewal immigration, where $F_{c(m)}$ is the interarrival distribution of the immigrants. This construction yields  $\mu(c(m))^{-1} = (1 - m)$ for the average intensity of the immigrants. Consequently, we obtain an average intensity $1$ for the resulting Hawkes processes $\{N^{(m)}\}_{ m\in[0,1)}$. In this construction, we may obtain arbitrary average intensities $\lambda>0$ by scaling the interarrivals of the immigrants by $\lambda^{-1}$.
\end{example}
The weak limit as $m\uparrow 1$ of the processes $\{N^{(m)}\}$ from Example~\ref{G:ex:1} can be studied in a similar manner as the limit of the construction in \citet{bremaud01}. The interpretation, however, is different. We provide another example starting at time 0, where we specify the tails of the distributions as regularly varying:
\begin{example}\label{G:ex:2}
For $i=1,2,$ let $F_i$ be a distribution on $[0, \infty)$ such that
$$
1 - F_i(x) \sim \frac{l_i(x)}{ x^{\alpha_i}\Gamma(1 + \alpha_i)},\quad x\to \infty,
$$
where $\alpha_i\in(0,1]$ and $l_i$ is slowly varying at infinity. We denote the renewal functions induced from $F_i$ by $U_i$. Consider a renewal process on $[0,\infty)$ with interarrival distribution $F_{1}$. From each renewal epoch, we start a Hawkes family process with displacement distribution $F_2$. Note that a generic Hawkes family (starting with an ancestor in 0) has mean measure
$
 U([0,x]) := \sum_{g\in\N_0} F^{g*}(x).
$ 
(This can be shown by considering the generations separately.)
For $x\in[0,\infty)$, denote the expected number of points  in $[0,x]$ of the resulting point process by $\overline{U}(x)$. For any non-negative function $G$ of bounded variation, we denote its Laplace--Stieltjes transform by
$\hat{G}:[0,\infty)\ni s \mapsto\int_{0}^\infty e^{-sx}\mathrm{d}G(x)\,(\in[0,\infty])$. We have that
$\overline{U} = U_1 * U_2$ and, consequently, 
\begin{align}
\hat{\overline{U}}(s)  =  \frac{1}{1 - \hat{F}_1(s)}\frac{1}{1 - \hat{F}_2(s)} \quad s\in [0,\infty).
\end{align}
From \citet{bingham87} (BGT), Corollary 8.1.7, we get for $i=1,2$ that $1 - \hat{F}_i(s) \sim s^{\alpha_i} l_i(1 / s)$ as $s\downarrow 0$, . Thus, we obtain
$$
\hat{\overline{U}}(s) \sim \frac{1}{s^{\alpha_1 + \alpha_2} l_1(1 / s)l_2(1 / s)},\quad s\downarrow 0.
$$
As $(l_1(\cdot)l_2(\cdot))^{-1}$ is again slowly varying, we may apply Karamata's Tauberian Theorem (Theorem 1.7.1 in BGT) to find
\begin{align}
\overline{U}(x) \sim \frac{x^{\alpha_1 + \alpha_2}}{ l_1(x) l_2(x)\Gamma(1 + \alpha_1 + \alpha_2)},\quad x\to\infty.
\end{align}
Thus, setting $\alpha_1:=\alpha\in[0,1)$ and $\alpha_2 := 1- \alpha$ and choosing $l_1$ and $l_2$ such that $\lim_{x\to\infty}(l_1(x)l_2(x))^{-1} = \lambda >0$, we get an `elementary renewal behaviour' for $\overline{U}$:
\begin{align}
\lim\limits_{x \to \infty}\frac{\overline{U}(x)}{x} =  \lim\limits_{x\to\infty} \frac{x^{\alpha + 1 - \alpha}}{ x l_1(x) l_2(x)\Gamma(1 + \alpha + 1 - \alpha)}=\lim\limits_{x\to\infty} \frac{1}{ l_1(x) l_2(x)}=\lambda.
\end{align}
That is, for all $\alpha\in[0,1)$, the averaged expectation of such critical Hawkes processes with renewal immigration on $[0,\infty)$ converges. Naturally, this does not imply distributional convergence. Furthermore, the symmetry in $i =1,2$ is only valid for the mean measure. Also note that $\alpha_1 = \alpha$ controls the interarrivals of the immigrant points and $\alpha_2 = 1 - \alpha$ controls the displacement of offspring points. Intuitively, we want the distance between the immigration points to be `larger' than between offspring points---otherwise, the offspring processes thin out faster than the immigrant process and thus vanish. In other words, for survival of the limit process, we will  need $\alpha_1 <\alpha_2$, equivalently, $\alpha < 1 - \alpha$ and, thus, $\alpha \in(0, 0.5)$---as in \eqref{G:eq:bremaud_conditions}. \end{example}

\subsection{Backward tree and Palm process}\label{G:sec:backward_tree}
We build a `Palm version' of a critical $(F,\lambda)$-Hawkes process by reconstructing the process starting from some fixed point. Here, the role of the symmetrized distribution $\tilde{F}$ is most obvious. We use an approach similar to the `method of backward trees'; see \citet{kallenberg77} or \citet[page 336]{daley03}. We pick an arbitrary point of a critical Hawkes process and shift the whole process in such a way that this arbitrary point has position 0. Obviously, from this point, we may walk back to its parent, then to its grandparent, etc. In other words, there is an infinite spine in the underlying (backwards) tree. In this way, we construct a Palm version of a critical $(F,\lambda)$-Hawkes process with (not necessarily absolutely continuous) displacement distribution $F$:
\begin{enumerate}[label = \alph*)]
\item  Fix a single special point at position 0 and attach a Hawkes family to it.
\item Generate its parent at position $-X_1$, with $X_1 \sim F$ and attach a Hawkes family to this parent.
\item Generate its grandparent at position $-X_1 - X_2$ with $X_2\sim F$ and attach Hawkes family to this grandparent.
\item Continue in this way.
\end{enumerate}
All random variables applied in this construction are chosen independent. As the construction is non-decreasing (when counting points in a fixed bounded Borel set), it will yield a random (possibly infinite) point measure $N_0$ on $\R$ with (possibly infinite) mean measure $U_0$. The limit process can be described as follows: the infinite spine of ancestors of the starting point forms a backward renewal process, with interarrival distribution ${F}_-$, where $F_-$ denotes the distribution function of the random variable $-X$, where $X\sim F$. This ancestor process has locally finite mean measure
$
 U_- := \sum_{g\in\N_0} F_-^{g*}.
$
(For any distribution $G$, we write $G((a,b]) := G(b) - G(a),\,a < b,$ and use the notion `$G$' as well for the distribution function as for the measure it defines.)
Each renewal point marks the start of a new (forward) Hawkes family.  
Thus, we obtain for the mean measure of the limiting process $N_0$
\begin{align}\label{G:eq:U_0}
U_0   & = U_{-} * U =
\sum_{g\in\N_0}\sum_{g'\in\N_0} F_-^{g*} *F^{g'*} = \sum_{g\in\N_0}\sum_{g' \geq g} F_-^{g*} *F^{g'*} + \sum_{g\in\N_0}\sum_{g' < g} F_-^{g*} *F^{g'*}.
\end{align}
Note that a random walk starting in 0 with step-size distribution $\tilde{F}$, the symmetrized version of $F$, has mean measure
\begin{align}
\tilde{U} :=  \sum_{g\in\N_0} \tilde{F}^{g*} =   \sum_{g\in\N_0} (F_{-}*{F})^{g*}  = \sum_{g\in\N_0} F_{-}^{g*} *{F}^{g*}.\label{G:eq:transience}
\end{align}
Consequently, for the first summand of the right-hand side in $\eqref{G:eq:U_0}$, we obtain
\begin{align*}
 \sum_{g\in\N_0}\sum_{g' \geq g} F_-^{g*} *F^{g'*}=\sum_{g\in\N_0}\sum_{g' \geq g} F_-^{g*} *F^{g*} * F^{*(g' - g)}
=
\sum_{g\in\N_0} F_-^{g*} *F^{g*} *\sum_{g' \geq g} F^{*(g' - g)}
=  \tilde{U}  *U .
\end{align*}
After a similar calculation for the second summand, we finally get that
\begin{align}
U_0 = \tilde{U}  *(U  +  U_- - \delta_0).\label{G:eq:palm_measure}
\end{align}
Note that $U_0$ is a symmetric measure (as a convolution of symmetric measures), and that $U_0(\{0\})\geq 1$---with equality if $F$ is absolutely continuous. These two facts were already visible in the first expression of \eqref{G:eq:U_0}. More importantly, we learn from \eqref{G:eq:palm_measure} that $U_0$ coincides with  the expectation of the occupation measures of infinitely many random walks, with each walk starting at renewal times of a two-sided renewal process (with a single renewal in 0). Local finiteness of $U_0$ means that the construction yields a point process with finite intensity whose law can be identified with a Palm distribution. We summarize:
\begin{conjecture}
If ${F}$ is transient, then $N_0$ defines a point process (that is, a random locally finite counting measure) with distribution $L_0$ and $\lambda L_0$ is the Palm measure of a stationary point process with finite average intensity $\lambda$---the critical $(\lambda,F)$-Hawkes process.
\end{conjecture}
We were not able to show local finiteness of the measure $U_0$ from $\eqref{G:eq:palm_measure}$ in such full generality. Instead, we supply another example in the case of a displacement distribution $F$ with regularly varying tails:
\begin{example}\label{G:ex:3}
Let $F(x)\sim x^{-\alpha} l(x), \, x\to \infty,$ with $l$ slowly varying at infinity and $\alpha\in(0,1]$. For all $h>0$, we have that
\begin{align}\label{G:eq:two_sided_renewal_convolution}
U_0([0,h]) = U_-*U([0,h])  &= \int_{(-\infty,0]} U\big([-x, h-x]\big) U_-(\mathrm{d} x)
=  \int_{[0, \infty)} U\big([x -h , x]\big) U(\mathrm{d} x),
\end{align}
where we use in the last equality that $U_-(B) = U(-B),\, B\in\mathcal{B}_b(\R).$ For $x\to\infty$, the integrand is of the same order as $x^{\alpha - 1}/l(x)$ (apart from a set of measure 0) by Theorem 8.6.6 in BGT. For the density $u$ of $U$ (assuming it exists and is ultimately monotone), we have that $x\sim \alpha l(x) x^{\alpha - 1},\, x\to\infty,$ by the Monotone Density Theorem; see Theorem~1.7.2 in BGT. Thus, for large $x$, the integrand with respect to Lebesgue measure is of the same or lower order as $x^{2(\alpha - 1)}$ for $x\to\infty$. Consequently, we find that $U_0$ defines a locally bounded measure if $\alpha \in (0, 0.5)$---as in \eqref{G:eq:bremaud_conditions}
\end{example}

\subsection{Kesten tree}
The constructions of the critical Hawkes process in Sections~\ref{G:sec:renewal_immigration} and~\ref{G:sec:backward_tree} are related to so-called \emph{Kesten trees} or \emph{size-biased trees}, a generalization of Galton--Watson trees, where the nodes are either of a \emph{normal} or of a \emph{special} type; see \citet{lyons95}. The distribution of the (independent) offspring operations of the nodes depends on their type. If $(p_k)$ is the offspring distribution of a normal node and $m\in(0,\infty)$ the corresponding expectation, then $(kp_k/m)$ is the \emph{size-biased offspring distribution} of a special node. In our case, where the normal offspring distribution is $\Pois(1)$, one can check that the size bias  corresponds to adding $+1$ to a $\Pois(1)$ random variable. The root node $\emptyset$ of a Kesten tree is special. Normal nodes have only normal children whereas special nodes have exactly one special child and otherwise normal children. Obviously, such a Kesten tree never dies out. In fact, it is well known that a Kesten tree is distributed like the corresponding Galton--Watson tree conditional on non-extinction. Denote such a Kesten tree (with respect to $\Pois(1)$ normal offspring and `$\Pois(1) + 1$' special offspring) by $\bf T$, its nodes by $\{\sigma\}_{\sigma\in \bf T}$, and its root node by $\emptyset$. We write $\sigma^-$ for the unique parent node of $\sigma\in\bf T\setminus\{\emptyset\}$. Every node is supplied with a position in $\R$ in a recursive (random) way:   
\begin{align}\label{G:eq:kesten_construction1}
S_\emptyset := 0,\quad 
S_\sigma :=\begin{cases} 
S_{\sigma^-} + Y^{(1)}_{\sigma},&\text{if $\sigma\in\bf T$ is a normal node and }\\
 S_{\sigma^-} + Y^{(2)}_{\sigma},&\text{if $\sigma\in\bf T$ is a special node and $\sigma\neq\emptyset$,}
 \end{cases}
\end{align}
where $Y^{(i)}_{\sigma}\sim F_i,\, i = 1,2,\,\sigma\in\bf T$, are independent. The distribution $F_1$ coincides with $F$, the displacement distribution; the distribution $F_2$, $F_2(0) = 0,$ controls the desired limiting average intensity $\lambda$.
The chain of nodes along the special nodes of $\bf T$ form an infinite spine. The position of these nodes along the infinite spine correspond to the immigrant renewal process from Section~\ref{G:sec:renewal_immigration}. Obviously, one could represent the specific constructions from Examples~\ref{G:ex:1} and \ref{G:ex:2} in a similar manner as \eqref{G:eq:kesten_construction1}. \par
Similarly, the backward or Palm construction from Section~\ref{G:sec:backward_tree} may be written in terms of Kesten trees. The underlying tree is exactly the same as for the case with renewal immigration in \eqref{G:eq:kesten_construction1}. However, the position labels change. Namely, we set
\begin{align}\label{G:eq:kesten_construction2}
S_\emptyset := 0,\quad 
S_\sigma :=\begin{cases} 
S_{\sigma^-} + Y_{\sigma},&\text{if $\sigma$ is a normal node and }\\
S_{\sigma^-} - Y_{\sigma},&\text{if $\sigma$ is a special node and $\sigma\neq\emptyset$,}
\end{cases}
\end{align}
where $Y_{\sigma}\sim F,\,$ iid, with $F$ the displacement distribution.
Thus, for studying the distribution of a critical Hawkes process, we may apply limit theorems for Kesten trees to its genealogical structure and then---with the tree given---study the positions separately by standard renewal or random walk theory.

\section{Discussion}
The presented methods open the door for discussions of multitype critical Hawkes processes as well as of critical integer-valued autoregressive time series:
\subsection{Critical multitype Hawkes processes}
The Poisson embedding, the connection to critical cluster fields, as well as the renewal immigration representation and the Palm construction allow for multitype generalizations. These generalizations will involve regime switching renewal processes and/or regime switching random walks. Algebraic properties of the branching matrix such as irreducability will be important. Results on multitype critical cluster fields as in \citet{ivanoff82} can be applied. A possible formalization for multiple size-biased trees is given in \citet{kurtz97}.

\subsection{Critical autoregressive time series}\label{G:sec:citical_INAR} In the same way as this paper analyzes critical monotype Hawkes processes, one could discuss critical univariate integer-valued autoregressive time series of infinite order---that is, critical INAR($\infty$) processes; see \citet{kirchner_hawkes_inar} for the subcritical case and for the INAR--Hawkes relation. More explicitely, one could consider time series $(X_n)$ solving
\begin{align}\label{G:eq:inar}
X_n = \sum_{k = 1}^\infty \sum_{l = 1}^{X_{n-k}} \xi^{(\alpha_k)}_{n,k}, \quad  \E X_n\equiv \lambda>0, \quad n\in\Z, 
\end{align}
for some independent random variables $\{\xi^{(\alpha_k)}_{n,k}\}$ with $\xi^{(\alpha_k)}_{n,k}\sim\Pois(\alpha_k),\,\alpha_k\geq0, k\in\N,\, n\in\Z$ and $\sum_{k\in\N}\alpha_k = 1$. Arguing in a similar manner as in Section \ref{G:sec:perspectives_on_critical_hawkes_processes} for the Hawkes process, we may rewrite \eqref{G:eq:inar} in terms of a (critical) cluster operation with offspring distribution $\Pois(1)$ and displacement distribution $(\alpha_k)$. The link between the clustering and the counting variables $(\xi_{ n,k})$ in \eqref{G:eq:inar} is provided by the property 
\begin{align}\label{G:eq:inar_counting}
\xi_{n,k} := \#\{Y_{n,l} = k:\, l = 1, \dots, K_n\} \sim \Pois(\alpha_k),\quad \text{independently over } k\in \N,\,n\in\Z,
\end{align}
when $\{K_n,Y_{n,l}:\, l\in\N,\, n\in\Z\}$ are independent random variables with $K_n\sim\Pois(1)$ and $Y_{n,l}\sim(\alpha_k)$. (In \eqref{G:eq:inar_counting}, we use the convention that $K_n = 0 \Rightarrow \xi_{n,k} = 0$.) In analogy to Theorem~\ref{G:thm:necessary_conditions}, one can then show that solutions to \eqref{G:eq:inar} necessarily specify a unique stationary time series distribution. Thus, our conjecture is that such a critical INAR($\infty$) process exists if and only if the symmetric random walk with step size distribution $(\sum^{\infty}_{l=1}\alpha_{l}\alpha_{k+l})_{k\in\Z}$ ($\alpha_k := 0, k\leq 0$) is transient. Note that we may then argue as after Conjecture~\ref{G:conj:1} that necessarily $\sum_{k = 1}^\infty\alpha_kk= \infty$. So in particular, we need $\alpha_k>0$, infinitely often, and a critical INAR($p$) process with $p <\infty$ cannot exist.
It is interesting to study the more familiar autoregressive representation of \eqref{G:eq:inar}, namely
\begin{align}\label{G:eq:ar}
X_n = \sum_{k = 1}^\infty \alpha_k{X_{n-k}} + u_n,\quad n\in\Z,
\end{align}
with $u_n := \sum_{k = 1}^\infty \sum_{l = 1}^{X_{n-k}} \xi^{(\alpha_k)}_{n,k}- \sum_{k = 1}^\infty \E \sum_{l = 1}^{X_{n-k}} \xi^{(\alpha_k)}_{n,k}$.  The innovations $(u_n)$ are stationary and have zero marginal means. In addition, one can show that $\Cov(u_n, u_{n'}) = 0,\, n\neq n'$. In other words, the critical INAR($\infty$) time series is a nontrivial critical (=`unit root') \emph{and stationary} autoregressive process. This stands in putative contradiction to standard time series theory; see e.g. Exercise 4.28.*  in \citet{brockwell91}. However, note that stationarity does not imply `weak stationarity', respectively, `covariance stationarity' (the meaning of `stationarity' in the mentioned exercise) if $\Var(X_n)=\infty$. We conclude that $X_n$ as in \eqref{G:eq:inar} or in \eqref{G:eq:ar} has infinite variance. Note that $\Var(u_n)<\infty$ might still be possible because the conclusion `$\Var(Z_1)=\Var(Z_2) = \infty\Rightarrow\Var(Z_1 + Z_2)=\infty$' is in general wrong. 

\subsection{Conclusion}
We have identified the distribution of a critical Hawkes process: it coincides with the distribution of a cluster-invariant point process. From Theorem~\ref{G:thm:necessary_conditions}\ref{G:item:dbn_limit}, we get that this distribution can be constructed by starting with a Poisson random field of ancestors, then applying iterated clustering---only considering children, then only grandchildren, then only great-grandchildren, etc. In other words, the points of a critical Hawkes process are related like `cousins of a very, very high degree'. 
\bibliographystyle{apa}
\bibliography{/Users/matthiaskirchner/Dropbox/ETH/papers/bibliographies/diss.biblio}
\end{document}